\documentclass[11pt]{article}
\usepackage{amsmath,amssymb}
\usepackage{amsthm,bm}
\usepackage[small]{titlesec}
\usepackage{relsize}
\usepackage{geometry}
\newtheorem{thm}{Theorem}[section]
\newtheorem{cor}[thm]{Corollary}
\newtheorem{lemma}[thm]{Lemma} 
\newtheorem{ex}[thm]{Example}

\newtheorem{defn}[thm]{Definition}
\newtheorem{rem}[thm]{Remark}

\numberwithin{equation}{section}

\newcommand{\punt}{\boldsymbol{.}}

\newcommand{\nor}{\mathcal N}

\newcommand{\invv}{\scriptscriptstyle{-1}}

\newcommand{\E}{\mathbb E}

\newcommand{\m}{\boldsymbol{m}}

\newcommand{\nub}{\boldsymbol{\nu}}
\newcommand{\etab}{\boldsymbol{\eta}}
\newcommand{\mub}{\boldsymbol{\mu}}

\newcommand{\ibs}{\boldsymbol{i}}
\newcommand{\jbs}{\boldsymbol{j}}
\newcommand{\timesmall}{\mathsmaller{\mathsmaller{\times}}}

\newcommand{\yp}
{\{y\}^{\scriptscriptstyle{1}}_{\scriptscriptstyle{p}}}
\newcommand{\xn}
{\{x\}^{\scriptscriptstyle{1}}_{\scriptscriptstyle{n}}}
\newcommand{\ynn}{\{y\}^{\scriptscriptstyle{1}}_{\scriptscriptstyle{2}}}
\newcommand{\xnn}{\{x\}^{\scriptscriptstyle{1}}_{\scriptscriptstyle{3}}}
\newcommand{\deltap}
{\{\delta\}^{\scriptscriptstyle{1}}_{\scriptscriptstyle{p}}}
\newcommand{\deltan}
{\{\tilde{\delta}\}^{\scriptscriptstyle{1}}_{\scriptscriptstyle{n}}}

\newcommand{\mm}{\boldsymbol{m}}

\newcommand{\X}{\boldsymbol{X}}

\newcommand{\trasp}{\scriptscriptstyle{\sf T}}

\newcommand{\Tr}{{\rm Tr}}

\begin{document}

\title{\Large Polynomial traces and elementary symmetric functions in the latent roots of a non-central Wishart matrix \thanks{The original source of publication is available in {\sl Journal of Multivariate Analysis}, 179 (2020) https://doi.org/10.1016/j.jmva.2020.104629}}

\author{Elvira Di Nardo \footnote{E. Di Nardo: Dipartimento di Matematica \lq\lq G.Peano\rq\rq, Universit\`a di Torino, Via Carlo Alberto 10, 10123, Torino, Italy, e-mail: \texttt{edinardo@unito.it}}}

\date{}

\maketitle

\begin{abstract}
Hypergeometric functions and zonal polynomials are the tools usually addressed in the literature to deal with the expected value of the elementary symmetric functions in non-central Wishart latent roots. The method here proposed recovers the expected value of these  symmetric functions by using  the umbral operator applied to the trace of suitable polynomial matrices and their cumulants. The employment  of a suitable linear operator in place of hypergeometric functions and zonal polynomials was  conjectured by de Waal in \cite{Dewaal}. Here we show how the umbral operator accomplishes this task and consequently represents an alternative tool to deal with these  symmetric functions.  When special formal variables are plugged in the variables,  the evaluation through the umbral operator deletes all the monomials in the latent roots except those contributing in the elementary symmetric functions. Cumulants further simplify the computations taking advantage of the convolution structure of the polynomial trace. Open problems are addressed at the end of the paper.
\end{abstract}
\\ \smallskip \noindent
{\it AMS Mathematics Classification:} {\small Primary: 60E10, 62H10} {\small Secondary: 05E05}
\section{Introduction\label{sec:1}}
The elementary symmetric function (e.s.f.) of degree~$i \le p$ in variables $y_1,\ldots, y_p$ is \cite{Macdonald}
\begin{equation}
e_{i}(y_1, \ldots, y_p)  = \!\!\!\!\! \sum_{1 \leq {j}_1 <  \cdots < {j}_{i} \leq p} \!\!\!\!\!\!\!\! y_{j_1} \times  \cdots \times y_{j_i} = \frac{1}{i !} B_{i}(g_1 s_1, \ldots, g_{i} s_{i}),  \quad  {i} \leq p
\label{(elfun)}
\end{equation}
where $B_{i}$ is  the ${i}$-th complete (exponential) Bell polynomial, $s_k = \sum_{i=1}^p y_{i}^k$ is the $k$-th power sum symmetric polynomial in $y_1, \ldots, y_p$ and $g_k=(-1)^{k-1} (k-1)!$ for all nonnegative integers $k.$ We also use the convention that $e_0 = 1.$

Suppose to replace $y_1, \ldots, y_p$ in \eqref{(elfun)} with the latent roots $Y_1, \ldots, Y_p$  of  a $p \times p$ random matrix $S$ and denote by $\Tr_i(S)$ its $i$-th e.s.f. in $Y_1, \ldots, Y_p.$ Taking the expectation $\E$ of both sides in \eqref{(elfun)}, we recover  $\E[\Tr_i(S)]$ through $\E[e_{i}(Y_1, \ldots, Y_p)].$ The usefulness of an explicit expression of $\E[\Tr_i(S)]$  essentially relies on the fundamental theorem on symmetric functions, as any symmetric polynomial has an expression in terms of the e.s.f.'s \cite{Stanley}. Moreover latent roots of random matrices are employed in various multivariate test procedures \cite{PARUCHURI}. More applications are given in \cite{Mathai} and references therein.
As outlined in  \cite{PARUCHURI}, the computation of $\E[\Tr_i(S)]$ using the latent root distribution can be a difficult task. A different 
procedure might consist in resorting joint moments $\E [\mu(S)(\tau)]=\E \prod_{c \in C(\tau)} \Tr\big(S^{{\mathfrak l}(c)}\big)$ of $S\!,$ introduced in \cite{CC}. Indeed, if ${\mathfrak S}_i$ denotes the symmetric group, a different expression of $e_{i}(Y_1, \ldots, Y_p)$  is (see the Appendix)
\begin{equation}
i! e_{i}(Y_1, \ldots, Y_p) =  \sum_{\tau \in {\mathfrak S}_i} (-1)^{i-|C(\tau)|} \prod_{c \in C(\tau)} \Tr\left(D_Y^{{\mathfrak l}(c)}\right)  
\label{(symPS1bis)}
\end{equation}
where the summation is over all permutations $\tau \in {\mathfrak S}_i$ of $[i] = \{1, \ldots,i\},$ $C(\tau)$ denotes the standard representation of $\tau$ in disjoint cycles, ${\mathfrak l}(c)$ is the cardinality of the cycle $c \in C(\tau)$ and $D_Y={\rm diag}(Y_1, \ldots, Y_p).$ 
Taking the expectation $\E$ of both sides in (\ref{(symPS1bis)}), we recover $i! \E[\Tr_i(S)]$ in terms of  joint moments  $\E [\mu(S)(\tau)].$  But also the computation of $\E [\mu(S)(\tau)]$ is a quite difficult task, even for well known matrix variate distributions, as the Wishart ones. Let us recall that the (non-singular) non-central Wishart random matrix of order $p$ is 
$$
W_p(n, \Sigma, M) =  {\mathsf X} {\mathsf X}^{\trasp} \qquad \hbox{with} \,\, n \geq p
$$
where ${\mathsf X}=(\X_1, \ldots, \X_n) \sim \nor_{p,n}(M, \Sigma, I_n)$ is a $p \times n$-matrix variate normal distribution with mean  $M = (\m_1, \ldots, \m_n),$ row covariance matrix $\Sigma$ and column covariance matrix $I_n,$ that is $\X_1, \ldots, \X_n$ are column random vectors independently drawn from a $p$-variate normal distribution $\X_i \sim {\mathcal N}(\m_i, \Sigma)$  with mean $\m_i \in {\mathbb R}^p$ and full rank covariance matrix $\Sigma$ of order $p.$ Note that $\Omega = \Sigma^{{\invv}} M M^{\trasp}$ is named non-centrality matrix.  

For $W=W_p(n,\Sigma,0)$ there are manageable closed form formulae for joint moments  $\E [\mu(W)(\tau)],$  see \cite{Wishartmio, Letac}. Hence, taking into account (\ref{(symPS1bis)}), an explicit expression of $\E \left[ \Tr_i(W) \right]$ might be recovered depending on $\E [\mu(W)(\tau)].$ But for $M \ne 0$ the computation of $\E [\mu(W)(\tau)]$ is quite cumbersome and in the literature a different way has been addressed. In particular,  using hypergeometric functions, zonal polynomials and the character of the symmetric group  ${\mathfrak S}_i,$ Shah and Khatri \cite{Shah} prove that
\begin{equation}
\E \left[ \Tr_i(W) \right] = (n)_i \Tr_i (\Sigma) + \sum_{k=1}^i (n-k)_{i-k} \sum_{j(i)} \det \big(\Sigma_{j(i)} \big) \Tr_j [\Sigma_{j(i)}^{\invv} (M M^{\trasp})_{j(i)}],
\label{(second)}
\end{equation}
where  the inner summation is over all the ordered $i$-tuples $j_1, \ldots, j_i$ of integers choosen in $\{1, \ldots, p\}$ and $\Sigma_{j(i)}$ and $ (M M^{\trasp})_{j(i)}$ are the principal submatrices formed with the $j_1, \ldots, j_i$-th rows and $j_1, \ldots, j_i$-th columns of $\Sigma$ and $M M^{\trasp}$ respectively. Special cases of \eqref{(second)} are \cite{Dewaal, Saw}
\begin{equation}
\E[\Tr_i (W)] \! = \! \left\{ \begin{array}{ll} 
(n)_i \Tr_i (\Sigma) & \hbox{if} \,\, M = 0 \vspace{0.3cm} \\
\sigma^{2i} \sum_{j=0}^i (n-j)_{i-j} \binom{p-j}{i-j} \Tr_j (M M^{\trasp})  & \hbox{if} \,\, \Sigma = \sigma^2 I_p \vspace{0.3cm} \\
\det(\Sigma) \sum_{j=0}^p (n-j)_{p-j} \Tr_j (\Omega) & \hbox{if} \,\, i=p 
\end{array} \right.
\label{(first)} 
\end{equation}
where $(n)_j= n (n-1) \cdots (n-j+1)$ for $1 \leq j \leq p \leq n.$ 
In this paper, our interest is focused on the latter conjecture formulated by de Waal in \cite{Dewaal}, about the existence of a suitable linear operator  providing \eqref{(first)} without using hypergeometric functions and zonal polynomials. 

Following his intuition, our aim is to recover  \eqref{(first)}, and then \eqref{(second)}, using the evaluation umbral operator introduced in \cite{SIAM} and the symbolic method \cite{DiNardo}.
If we wanted to apply the classical umbral calculus \cite{SIAM} plainly, we would have to use the algebra of formal power series and the characteristic function of the latent roots of $W,$  see Section~\ref{sec:3}. Unfortunately this characteristic function \cite{Smith} has a quite cumbersome expression to be expanded in formal power series. Therefore in this paper, we propose to use the symbolic method \cite{DiNardo} arising from  the umbral calculus and involving the algebra of cumulant polynomials \cite{CumDin}. Despite its algebraic flavor, the method is also known as algebra of probability \cite{Twelve} since its syntax matches the one of random vectors. For readers who are unaware of the method, a short introduction is given in Sections ~\ref{sec:2} where the algebra of probability  is addressed in terms of umbrae and where we recall notations and definitions needed to work with. The extension to the multivariate framework is recalled in Section~\ref{sec:3} (for more details see  \cite{DiNardo} and references therein). 

Let us underline that the symbolic method was already employed for computing moments and cumulants of $\Tr({\mathsf X} {\mathsf X}^{\trasp})$ by using Sheffer polynomials, see \cite{Wishartmio}.  Therefore the contents of this paper represent a prosecution 
of \cite{Wishartmio} with one more application.  We introduce a new class of polynomials, the polynomial trace $\Tr (D_y {\mathcal X} D_x)(D_y {\mathcal X} D_x)^{\trasp},$ where $D_x$ and $D_y$ are diagonal matrices of $n$ and $p$ indeterminates respectively and ${\mathcal X}$ is a suitable {\it formal} matrix  {\it symbolizing} ${\mathsf X}.$  Formulae \eqref{(second)}  and  \eqref{(first)} are then recovered by using this new class of polynomials and by taking advantage of the additivity property of their cumulants arising from the convolution between the central component of $\Tr (D_y {\mathcal X} D_x)(D_y {\mathcal X} D_x)^{\trasp}$  and the trace of a formal matrix involving $M$ and $\Sigma.$ Note that the symbolic calculus for cumulants of random matrices has been developed in \cite{annali}.  As shown in the last section, when we replace the indeterminates of $D_x$ and $D_y$ with suitable umbrae and evaluate the resulting polynomial through the umbral operator, only the monomials contributing in \eqref{(elfun)} give not zero contribution.  The same strategy has been already applied to recover different families of symmetric polynomials, as for example
the product of augmented polynomials in separately independent and identically distributed random variables \cite{DiNardo}. Open problems are addressed at the end of the paper.

\section{The moment symbolic calculus\label{sec:2}}
Denote by ${\mathcal A}=\{\alpha, \gamma, \ldots\}$ an alphabet of symbols called umbrae. The evaluation (umbral) linear operator $E$ is defined on the polynomial ring ${\mathbb R}[{\mathcal A}],$ with values in ${\mathbb R}$ and such that $E[1]=1$ and 
\begin{enumerate}
\item[{\rm (i)}] $E(\alpha^i) = a_i$ for all nonnegative integers $i$  with $a_0=1$,
\item[{\rm (ii)}] $E(\alpha^i \gamma^j \times \cdots) = E(\alpha^i) E(\gamma^j) \times \cdots$ for distinct umbrae $\alpha, \gamma, \ldots$ and nonnegative integers $i,j,\ldots$ (uncorrelation property).
\end{enumerate}
The sequence $\{a_i\}_{i \geq 0}$  is said to be umbrally represented by $\alpha$, and $a_i$ is called the $i$-th moment of the umbra $\alpha.$ Distinct symbols of ${\mathcal A}$ denote uncorrelated umbrae. Two umbrae $\alpha$ and $\gamma$ are said to be similar iff $E(\alpha^i) = E(\gamma^i)$ for all nonnegative integers $i,$ in symbols $\alpha \equiv \gamma.$
By extending coefficientwise the operator $E$ to the ring of formal 
power series ${\mathbb R}[[z]],$ the generating function (g.f.) of $\alpha$ is the formal power series
\begin{equation}
f(\alpha,z)= E \big( e^{\alpha z} \big) = \sum_{i \geq 0}  a_i \frac{z^i}{i!} \in {\mathbb R}[[z]].
\label{(genfun)}
\end{equation}
Thus, $\alpha \equiv \gamma$ iff $f(\alpha,z)=f(\gamma,z)$ and the alphabet ${\mathcal A}$ can be endowed with sufficiently many umbrae similar with any expression whatsoever \cite{SIAM}.  Moreover the formal power series  \eqref{(genfun)} needs not have a convergence region \cite{Stanley}.

A random variable (r.v.) with  all moments $\{\E(X^i)\}_{i \geq 0}$ is represented by an umbra $\alpha$ having the same moments. In particular if $X$ admits moment generating function (m.g.f.) ${\mathcal M}_X(z),$ then $X$ is represented by an umbra $\alpha$ with $f(\alpha,z)={\mathcal M}_X(z).$ For example, the r.v. $X$ such that ${\mathbb P}(X=1)=1$ has all moments equal to $1$ and is represented by the unit umbra $u$ with g.f. $f(u,z)=\exp(z).$ The Poisson r.v. $\mathcal{P}(1)$ is represented by the Bell umbra $\beta$ such that $f(\beta,z)=\exp[\exp(z-1)].$ If $X$ is a r.v. with m.g.f. ${\mathcal M}_{X}(z),$  any polynomial $p(X)$ with m.g.f. ${\mathcal M}_{p(X)}(z)$ can be represented by $p(\alpha),$ where $\alpha$ is an umbra representing $X.$ For a discussion on the employment of formal power series in dealing with a finite sequence of moments see \cite{Taqqu}.

The correspondence between umbrae and r.v.'s is not one-to-one. Despite the evaluation operator resembles  the expectation of a r.v., the formal variables $\alpha, \gamma, \ldots$  need not have a probabilistic counterpart. For example, the sequence $\{1,1,0,\ldots\}$ is represented by the so-called singleton umbra $\chi,$ with g.f. $f(\chi,z)=1+z,$ which does not have a probabilistic counterpart. Moreover there are r.v.'s that cannot be represented by an umbra as they do not have moments. These issues and other noteworthy probabilistic aspects of the umbral calculus have been developed in \cite{DiBucchianico}. Two umbrae will play a special role in the following: the singleton umbra and the delta umbra $\delta,$ such that $\delta^2 \equiv \chi$ and $f(\delta,z)=1+\frac{z^2}{2}.$  

\begin{ex}[Normal r.v.]\label{ex2.1}
{\rm Suppose $X \sim {\mathcal N}(0, 1)$ a standard normal r.v. As ${\mathcal M}_X(z) = \exp \left[f(\delta,z)-1\right]$ is the composition of $\exp[\exp(z-1)]$ and $1+ \log \left[f(\delta,z)-1\right]\!,$ then $X$ is represented by the  $\delta$-exponential auxiliary umbra $\beta \punt \delta.$ To simplify the notation, we denote $\beta \punt \delta$ by $\zeta.$ Thus the r.v.$\,X \sim {\mathcal N}(m, \sigma^2)$ is represented by $m u + \sigma \zeta$ as
$f\left(m u + \sigma \zeta, z \right)  =  \exp \big( m z + \frac{1}{2} \sigma^2 z^2 \big).$ A non-central chi-squared r.v. $\!\!\!$ with degree of freedom $1$ and non-centrality parameter $m$ is represented by $(m u + \zeta)^2.$}
\end{ex}
\subparagraph{Elementary symmetric functions.}
Details on this symbolic calculus for symmetric functions are given in \cite{DiNardo}. Here we just recall the results we need in the following. Suppose $\{q_i\}_{i \geq 0}$ a sequence of polynomials in  $y_1, \ldots, y_p$ such that $q_0=1$ and $\hbox{deg}(q_i)=i$ for all positive integers $i.$ To represent  such a sequence with an umbra, we replace the field ${\mathbb R}$  with the ring of polynomials ${\mathbb R}[y_1, \ldots, y_p]$ and consider the evaluation operator $E: {\mathbb R}[y_1, \ldots, y_p][{\mathcal A}] \mapsto {\mathbb R}[y_1, \ldots, y_p]$ such that
$$
E \left( y_i^{l_1} y_j^{l_2} \times \cdots \times \nu^{m_1} \mu^{m_2} \times \cdots \right) = y_i^{l_1} y_j^{l_2} \times \cdots \times E \left( \nu^{m_1} \mu^{m_2} \times \cdots \right) 
$$
for all $\nu, \mu, \ldots \in {\mathbb R}[{\mathcal A}],$ 
for all $i,j,\ldots \in \{1, \ldots,p\}$  and for all nonnegative integers $l_1,l_2,\ldots, m_1,m_2, \ldots.$ The umbra representing the polynomial sequence  $\{q_i\}_{i \geq 0}$ is said to be polynomial.
The elementary symmetric polynomial umbra  
$\epsilon(y_1, \ldots,y_p) = \chi_1 y_1 + \cdots + \chi_p y_p$ is an example as its moments are the e.s.f.'s in $y_1, \ldots, y_p$
\begin{equation}
E[(\chi_1 y_1 + \cdots + \chi_p y_p)^i] = \left\{ \begin{array}{ll}
i! e_i(y_1, \ldots, y_p), &  i=0, \ldots, p, \\
0, & i > p.
\end{array} \right.
\label{(efsy1)}
\end{equation} 
The g.f. is $f(\chi_1 y_1 + \cdots + \chi_p y_p, z)= \prod_{j=1}^p (1 + y_j z) .$ Notice that $\epsilon(1, \ldots, 1) = \chi_1 + \cdots + \chi_p$ has moments $(p)_i$ for $i \leq p$ and $0$ otherwise. To lighten the notation, in the following we denote $\chi_1 + \cdots + \chi_p$ by the auxiliary umbra $p \punt \chi,$ see  \cite{DiNardo} for more properties on auxiliary umbrae. 
\begin{ex}[$U$-statistics]\label{rem3.4}{\rm
If we plug $\alpha$  in $y_1, \ldots, y_p,$ from \eqref{(efsy1)} we have  $i! E[e_i(\alpha, \ldots, \alpha)] = E(\alpha^i) E[\epsilon(1, \ldots, 1)^i] = a_i (p)_i$ for $i \leq p.$ If we plug distinct umbrae  $\alpha_1, \ldots, \alpha_p$ similar to $\alpha$ in $y_1, \ldots, y_p,$ then $i! E[e_i(\alpha_1, \ldots, \alpha_p)]$ allow to recover the $U$-statistics of a random sample \cite{DiNardo}.}
\end{ex}
\section{The multivariate framework\label{sec:3}}
To represent random vectors, we consider umbral polynomials \cite{DiNardo}. A multi-indexed sequence $\{g_{\ibs}\}$ with $g_{\ibs} = g_{i_1, \ldots,i_p}$ and $\ibs=(i_1, \ldots,i_p) \in {\mathbb N}_0^p$ is represented by a $p$-tuple
$\nub=(\nu_1, \ldots, \nu_p)$ of umbral polynomials $\nu_1, \ldots, \nu_p \in 
{\mathbb R}[{\mathcal A}],$ if $g_{\boldsymbol{0}}=1$ and $E(\nub^{\ibs})=g_{\ibs}$ for all $\ibs \in {\mathbb N}_0^p.$
  By extending coefficientwise the evaluation $E,$ the g.f. of $\nub$ is
\begin{equation}
f(\nub,\bm{z}) = E[ \exp (\nu_1 z_1 + \cdots + \nu_p z_p  )] = \sum_{k \geq 0} \sum_{|\ibs|=k} \frac{g_{\ibs}}{\ibs!} \bm{z}^{\ibs} \in {\mathbb R}[[\bm{z}]],
\label{(genmult)}
\end{equation}
where $\bm{z}=(z_1, \ldots, z_p), |\ibs| = i_1 + \cdots + i_p,$ and $\ibs!=i_1! \times \cdots \times i_p!.$ If $g_{\ibs}$ is the $\ibs$-th multivariate moment of a (column) random vector $\X$
or ${\mathcal M}_{\X}(\bm{z})$ admits power series expansion \eqref{(genmult)}, then $\X$ is said umbrally 
represented by $\nub$  and $g_{\ibs}$ is  the $\ibs$-th multivariate moment of $\nub.$ Consider $\mub=(\mu_1, \ldots, \mu_p)$
with $\mu_1, \ldots, \mu_p \in {\mathbb R}[{\mathcal A}].$
\begin{defn}\label{2.2}
$\nub$ and $\mub$ are said to be similar if $f(\nub,\bm{z})=f(\mub,\bm{z}),$
in symbols $\nub \equiv \mub.$
\end{defn}
To represent not independent r.v.'s, we use {\it related} umbral polynomials. Let us recall that, when an umbral polynomial $\nu$ is written as a linear combination of distinct monomials with not zero coefficients, its support supp$(\nu)$ is the set of all umbrae that occur in some such monomial with a positive power \cite{SIAM}. A set of umbral polynomials with supports of any two of them disjoint is said to be unrelated (otherwise related).
\begin{defn}\label{2.3}
$\nub$ and $\mub$ are said to be unrelated if $\hbox{\rm supp}(\nub) = \cup_{i=1}^p \hbox{\rm supp}(\nu_i)$  is disjoint with $\hbox{\rm supp}(\mub) = \cup_{i=1}^p \hbox{\rm supp}(\mu_i).$
\end{defn}
If $\nub $ and $\mub$ are unrelated then $E[\nub^{\ibs} \mub^{\jbs}] = E[\nub^{\ibs}]E[\mub^{\jbs}]$ for all $\ibs,\jbs \in {\mathbb N}_0^p$ and in particular $f(\mub + \nub, \bm{z})= f(\mub, \bm{z})f(\nub, \bm{z}).$ 
\begin{ex}[Normal umbral $p$-tuples] \label{ex22} {\rm
Suppose $\X \sim \nor_p(\mm,\Sigma)$ having m.g.f. ${\mathcal M}_{\X}(\bm{z})  =  \exp( \bm{z} \mm  + \frac{1}{2} \bm{z} \Sigma \bm{z}^{\trasp}).$ Thus $\X$ is represented by an umbral $p$-tuple $\etab$ having g.f. ${\mathcal M}_{\X}(\bm{z})$ and we write $\etab \equiv  \nor_p(\mm,\Sigma).$ Recall that, if $C$ and $D$ are two $p \times q$-matrices, then the Hadamard product $C \circ D$ is the $p \times q$ matrix such that  $(C \circ D)_{ij} = (C)_{ij}(D)_{ij}$ for $i=1,\ldots,p$ and $j=1,\ldots,q.$ Thus  
$\nor_p(\mm,\Sigma) \equiv \mm^{\trasp} \circ \bm{u} \,  + \bm{\zeta} \Sigma^{1/2},$ where $\bm{u}$ is a $p$-tuple of distinct unity umbrae, $\bm{\zeta}$  is a $p$-tuple of distinct $\delta$-exponential umbrae and $\circ$ is the Hadamard product.
 Indeed as $\bm{u}$  and $\bm{\zeta}$ are unrelated, we have $f( \mm^{\trasp} \circ \bm{u} + \bm{\zeta} \Sigma^{1/2}, \bm{z}) = f(\mm^{\trasp} \circ \bm{u}, \bm{z}) f\big(\bm{\zeta} \Sigma^{1/2}, \bm{z} \big)$ with $f(\mm^{\trasp} \circ \bm{u}, \bm{z}) = \exp \left( \bm{z} \mm \right)$ and  $f(\bm{\zeta} \Sigma^{1/2}, \bm{z}) = f(\bm{\zeta}, \bm{z} \Sigma^{1/2}) = \exp (\frac{1}{2} \bm{z}\Sigma  \bm{z}^{\trasp} ).$}
\end{ex}
\begin{ex}[Generalized non-central chi-squared r.v.]\label{ex3.4}{\rm Suppose $\X \sim \nor_p(\mm,\Sigma)$ and consider $\Tr(\X \X^{\trasp})$ $= \X^{\trasp} \X,$ with degree of freedom $p$ and non-centrality parameter $\Sigma^{-1} \bm{m} \bm{m}^{\trasp}.$ Its m.g.f. is \cite{Gupta} 
\begin{align} 
{\mathcal M}_{\X^{\trasp} \X}(z)  & = \det{(I_p - 2 z \Sigma)}^{-1/2} \exp\left\{\frac{1}{2} \Tr \big[ \big((I_p - 2 \Sigma z)^{-1} - I_p\big) \Sigma^{-1} \bm{m} \bm{m}^{\trasp} \big] \right\} \nonumber \\
& = \exp \left[ \frac{1}{2} \sum_{k \geq 1} \frac{2^k z^k}{k}\Tr \big(\Sigma^k + \,\, k \Sigma^{k-1} \mm \mm^{\trasp}\big)\right]
\label{(fsr)} 
\end{align} 
where the function at the r.h.s. of \eqref{(fsr)} is obtained using the well-known equations
\begin{equation}  
\det{(I_p - z A)}^{-1}=\exp \bigg( \sum_{k \geq 1} \frac{z^k}{k} \Tr(A^k) \bigg) \,\, \,\, \,\, {\rm and} \,\, \,\, \,\,  (I_p - z A)^{-1} = \sum_{k \geq 0} z^k A^{k}
\label{(omsum1)}
\end{equation}
with $A$ a $p \times p$ matrix. If $\etab  \equiv \nor_p(\mm,\Sigma),$ then $\X^{\trasp} \X$ is represented  by $\etab \etab^{\trasp}.$ Indeed, suppose $\Sigma=Q D_{\theta} Q^{\trasp},$ with eigenvalues $D_{\theta} = {\rm diag}(\theta_1, \ldots, \theta_p),$  and observe that  $\etab \equiv  (\tilde{\mm}^{\trasp} \circ \bm{u} \,  + \bm{\zeta}  Q)  D_{\theta}^{1/2} Q^{\trasp}$ with $\tilde{\mm}^{\trasp}=\mm^{\trasp} Q D_{\theta}^{-1/2}$ from Example  \ref{ex22}.  As $f\left(\bm{\zeta} Q^{\trasp}, \bm{z} \right)  = \exp (\frac{1}{2} \bm{z} \bm{z}^{\trasp})$ then $\etab \equiv (\tilde{\mm}^{\trasp} \circ \bm{u} \,  + \tilde{\bm{\zeta}})  D_{\theta}^{1/2} Q^{\trasp}$ 
with $\tilde{\bm{\zeta}}$ a $p$-tuple of distinct $\delta$-exponential umbrae. Thus $\etab \etab^{\trasp} \equiv
\sum_{i=1}^p \theta_i (\tilde{m}_{i} u_i + \tilde{\zeta}_{i})^2$
and $f (\etab \etab^{\trasp}, z) = \prod_{i=1}^p f[(\tilde{m}_{i} u_i + \tilde{\zeta}_{i})^2, \theta_i z] = \exp \big( \frac{1}{2} \sum_{k \geq 1} \frac{2^k z^k}{k} \sum_{i=1}^p   \theta_i^k \left( 1 +  k \tilde{m}_{i}^2 \right) \big).$ Last equality follows from Example \ref{ex2.1}, using the uncorrelation property. As
$\sum_{i=1}^p   \theta_i^k \left( 1 +  k \tilde{m}_{i}^2 \right) = \Tr(D_{\theta}^k) + k \Tr( \tilde{\mm}^{\trasp} D_{\theta}^k  \tilde{\mm}^{\trasp})$ the g.f. $f(\etab \etab^{\trasp},z)$ matches the function on the r.h.s. of \eqref{(fsr)}.} 
\end{ex}
\begin{lemma}
If $(Y_1, \ldots, Y_p)$ is represented by the $p$-tuple $\nub,$ then 
$$i! \E[e_i(Y_1, \ldots, Y_p)] = E[p_i(\chi_1, \ldots, \chi_p)]$$ where  $p_i(y_1, \ldots, y_p)$ is the $i$-th coefficient of  $E[\exp\big( (y_1 \nu_1 + \cdots + y_p \nu_p) z \big)].$
\end{lemma}
\begin{proof}
 Indeed from \eqref{(efsy1)} we have $E[e_i(\nu_1, \ldots, \nu_p)]=E[p_i(\chi_1, \ldots, \chi_p)],$ 
that gives $\E[e_i(Y_1, \ldots, Y_p)]$ since the joint moments of $\nub$ are equal to the joint moments of $(Y_1, \ldots, Y_p).$
\end{proof}
\vskip2pt \indent
From the previous Lemma, a way to recover $\E \left[ \Tr_i(W) \right]\!$ is the following: compute  the $i$-th coefficient of  $\exp (y_1 \nu_1 + \cdots + y_p \nu_p) z,$ where the $p$-tuple $(\nu_1, \ldots, \nu_p)$ represents the latent roots of 
$W;$ plug $\{\chi_i\}$ in place of $\{y_i\}$ and evaluate the corresponding umbral polynomial  through $E.$  As remarked in the introduction, the m.g.f. of the latent roots of $W$ has a cumbersome expression to be expanded in formal power series \cite{Smith} and thus we recover  $\E \left[ \Tr_i(W) \right]$ following a different way. 

\section{Umbral matrices\label{sec:4}}
Let us consider a sequence $\{g_{\ibs_1, \ldots, \ibs_n}\}$ indexed by $n$  multi-indexes $\ibs_1, \ldots, \ibs_n \in {\mathbb N}^p_0$ and umbrally represented by a $n$-tuple ${\bm{\mathcal V}} = (\nub_1, \ldots, \nub_n)$ of umbral $p$-tuples. Paralleling (\ref{(genmult)}), the g.f. of ${\bm{\mathcal V}}$ is 
$$
 f \big({\bm{\mathcal V}};\bm{z}_1, \ldots, 
\bm{z}_n \big) = \sum_{k \geq 0} \sum_{|\ibs_1|+ \cdots + |\ibs_n|=k} g_{\ibs_1, \ldots, \ibs_n} \frac{\bm{z}_1^{\ibs_1} \times \cdots \times \bm{z}_n^{\ibs_n}}{\ibs_1! \times \cdots
\times \ibs_n!} \in {\mathbb R}[[\bm{z}_1, \ldots, \bm{z}_n]],
$$
with $\bm{z}_j = (z_{1j}, \ldots, z_{pj})$ for $j=1, \ldots,n.$ 
Definitions \ref{2.2} (similarity) and \ref{2.3} (unrelation) are naturally extended to $n$-tuples of umbral $p$-tuples.
\begin{ex}[Normal umbral $pn$-tuples] \label{ex3.7}  {\rm Let ${\rm vec}_{[p \times n]}$ be the operator representing a $p \times n$-matrix as a $pn$-vector formed by putting columns underneath starting with the first. In the following, we omit the subscript ${p \times n}$ in ${\rm vec}_{[p \times n]}$ when there are no misunderstandings. The $p \times n$-matrix variate normal distribution  ${\mathsf X} \sim \nor_{p,n}(M, \Sigma, \Psi)$ has m.g.f. \cite{Kollo}
\begin{equation}
M_{{\mathsf X}}(Z) = M_{{\rm vec}({\mathsf X})}(Z) = \exp \left[ \hbox{\rm vec}^{\trasp}(M) \hbox{\rm vec}(Z) \!  + \frac{1}{2} \hbox{\rm vec}^{\trasp}(Z)  (\Psi \otimes \Sigma) \hbox{\rm vec}(Z) \right]
\label{mgfnorm1}
\end{equation}
where  $\otimes$ denotes the Kronecker product, that is 
$\Psi \otimes \Sigma$ is the $pn \times pn$-matrix of  $p \times p$ block submatrices $[\Psi_{ij} \Sigma],$ for $i = 1, \ldots , n$ and $j = 1, \ldots, n.$ From Example \ref{ex22}, ${\rm vec}({\mathsf X}) \sim \nor_{pn}(\hbox{\rm vec}(M), \Psi \otimes \Sigma)$  is represented by the umbral $pn$-tuple  $\hbox{\rm vec}^{\trasp}(M) \circ  (\bm{u}_1, \ldots,\bm{u}_n)   + (\bm{\zeta}_1, \ldots, \bm{\zeta}_n) (\Psi^{1/2} \otimes \Sigma^{1/2}),$ where $\bm{u}_1, \ldots, \bm{u}_n$ are unrelated $p$-tuples of unity umbrae  and $\bm{\zeta}_1, \ldots,\bm{\zeta}_n$ are unrelated $p$-tuples of $\delta$-exponential umbrae. If $\Psi=I_n$ then ${\rm vec}({\mathsf X})$ is represented by $(\etab_1, \ldots, \etab_n),$ with $\etab_j \equiv \mm_j^{\trasp} \circ \bm{u}_j \,  + \bm{\zeta}_j \Sigma^{1/2}$
unrelated normal umbral $p$-tuples for $j=1, \ldots,n.$}
\end{ex}
\begin{defn} \label{fund}
A  $p \times n$-umbral matrix  is ${\mathcal V} = {\rm vec}^{{\scriptscriptstyle -1}}[ (\nub_1^{\trasp}, \ldots, \nub_n^{\trasp})^{\trasp}].$ 
\end{defn}
If $\hbox{\rm etr}( \cdot) = \exp \big[ \Tr ( \cdot) \, \big]$ and $Z={\rm vec}^{{\scriptscriptstyle -1}} [(\bm{z}_1^{\trasp}, \ldots,\bm{z}_n^{\trasp})^{\trasp}],$ then
$$
f(\bm{\mathcal V};  \bm{z}_1, \ldots, 
\bm{z}_n) = E \big\{ \exp \big[ \hbox{\rm vec}^{\trasp}({{\mathcal V}})  {\rm vec}(Z) \big] \big\} = E \left[ \hbox{\rm etr} \big( {{\mathcal V}}^{\trasp} Z\big) \right]
$$
and we set 
$$
f({\mathcal V},Z) = E \left[ \hbox{\rm etr} \big( {\mathcal V}^{\trasp} Z\big) \right].$$
If $V$ is a random matrix with m.g.f. $\E \left[ \hbox{\rm etr} \big( V^{\trasp} Z\big) \right] = M_V(Z),$  then $V$ is represented by the umbral matrix ${\mathcal V}$ with g.f. $f({\mathcal V},Z) = M_V(Z).$
\begin{ex}[Normal umbral matrix]\label{4.3}{\rm
 From Example \ref{ex3.7}, the matrix variate ${\mathsf X} \sim \nor_{p,n}(M, \Sigma, \Psi)$  is represented by a
$p \times n$-umbral matrix ${\mathcal X}$ having g.f. \eqref{mgfnorm1} and named normal umbral matrix. We write ${\mathcal X} \equiv \nor_{p,n}(M, \Sigma, \Psi).$ It's straightforward  to prove that $M \circ {\mathcal U} + \Sigma^{1/2}  {\mathcal Z} \Psi^{1/2}$ is a normal umbral 
matrix where ${\mathcal U}={\rm vec}^{-1} \left[(\bm{u}_1^{\trasp}, \ldots,\bm{u}^{\trasp}_n)^{\trasp} \right]$ and ${\mathcal Z}={\rm vec}^{-1}\left[(\bm{\zeta}_1^{\trasp}, \ldots, \bm{\zeta}_n^{\trasp})^{\trasp} \right].$} 
\end{ex}
\begin{ex}[E.s.f. and the singleton umbral matrix]
 Let us consider the $i$-th e.s.f. $e_i(\theta_1, \ldots, \theta_p)$ in the eigenvalues of $\Sigma$ and $D_{\theta} = {\rm diag}(\theta_1, \ldots, \theta_p).$ Then $\Tr_i(\Sigma) = E[\Tr({\mathfrak X}_p D_{\theta})^i]$ for $i \leq p,$ where ${\mathfrak X}_p = {\rm diag} (\chi_1, \ldots, \chi_p)$ is a diagonal singleton umbral matrix. As $E[\det({\mathfrak X}_p)]=1,$  notice that $E[\Tr({\mathfrak X}_p D_{\theta})^p] = \det(\Sigma) =  E[ \det({\mathfrak X}_p \Sigma)] = E[\Tr({\mathfrak X}_p
 \Sigma)^p].$ For $i \leq p,$ as ${\mathfrak X}_p \equiv \Delta_p^2$  with $\Delta_p={\rm diag}(\delta_1, \ldots, \delta_p)$ a diagonal delta umbral matrix, we have 
$$\Tr_i(\Sigma) = E[\Tr(\Delta_p D_{\theta} \Delta _p)^i] = 
E[\Tr(\Delta_p Q^{\trasp} \Sigma Q \Delta_p )^i] =  E[\Tr(Q_{\delta} \Sigma Q_{\delta}^{\trasp})^i]$$
where $Q_{\delta}=\Delta_p Q^{\trasp}.$ Notice that
$Q_{\delta} Q_{\delta}^{\trasp} \equiv {\mathfrak X}_p.$  
\end{ex}
\begin{ex}[E.s.f. and the Wishart umbral matrix]\label{4.4}{\rm 
The Wishart random matrix $W$ has m.g.f. \cite{Gupta} 
$$
{\mathcal M}_{W}(Z) =  \det \big(I_{p} - 2 \, Z  \Sigma\big)^{-n/2}  \hbox{\rm etr} \bigg[ (I_{p} - 2 Z \Sigma)^{-1} Z \Omega \bigg],$$
where  $Z$ is a $p \times p$ parametric matrix such that $Z_{ij} = \frac{1}{2} z_{ij}$ for $i \ne j$ and $z_{ij} = z_{ji} $ for $i,j \in \{1, \ldots, p\}.$ A formal power series expansion of ${\mathcal M}_{W}(Z)$ is given in Theorem 7.8.1. of \cite{Gupta}. By using the same arguments of Example \ref{ex3.4}, this formal power series results to be the g.f. of ${\mathcal X}{\mathcal X}^{\trasp},$ with ${\mathcal X} \equiv M \circ {\mathcal U} + \Sigma^{1/2}  {\mathcal Z} \Psi^{1/2}$. Then ${\mathsf X}{\mathsf X}^{\trasp}$ is represented by the Wishart umbral matrix ${\mathcal X}{\mathcal X}^{\trasp}.$  Suppose to represent the singular values of ${\mathsf X }$ with the umbral $p$-tuple $\nub=(\nu_1, \ldots, \nu_p)$ and with ${\mathcal V}$ the $p \times n$ rectangular diagonal matrix containing  $\nu_1, \ldots, \nu_p$ in the elements with equal indices. Thus $\nu_1^2, \ldots, \nu_p^2$ represent the latent roots of ${\mathsf X}{\mathsf X}^{\trasp}$ and $\Tr ({\mathfrak X}_p {\mathcal V}^2) \equiv \Tr(\Delta_p {\mathcal V} \tilde{\Delta}_n)(\Delta_p {\mathcal V} \tilde{\Delta}_n)^{\trasp},$ where $\Delta_p = {\rm diag}(\delta_1, \ldots, \delta_p)$ and $\tilde{\Delta}_n = {\rm diag}(\tilde{\delta}_1, \ldots, \tilde{\delta}_n)$ are diagonal matrices of uncorrelated delta umbrae such that 
$\delta_i^2 \equiv \chi_i$ for $i=1, \ldots,p$ and $\tilde{\delta}_i^2 \equiv \tilde{\chi}_i$ for $i=1, \ldots,n.$  This last formula has suggested the introduction of the polynomial trace umbra in order to recover  $\E[\Tr_i(W)].$}
\end{ex}
\subsection{Polynomial trace umbra}
Let us consider two sets of indeterminates $y_1, \ldots, y_p$ and $x_1, \ldots, x_n,$ and the evaluation operator $E$  defined on ${\mathbb R}[y_1, \ldots, y_p; x_1, \ldots, x_n][{\mathcal A}]$ such that
$$
E[y_{i_1}^{l_1} y_{i_2}^{l_2} \times  \cdots \times x_{j_1}^{k_1} x_{j_2}^{k_2} \times \cdots \times \nu^{m_1} \mu^{m_2} \times \cdots] = y_{i_1}^{l_1} y_{i_2}^{l_2} \times \cdots \times x_{j_1}^{k_1} x_{j_2}^{k_2} \times \cdots \times E[\nu^{m_1} \mu^{m_2} \times \cdots] 
$$
for all $\nu, \mu, \ldots \in {\mathbb R}[{\mathcal A}],$ for all 
nonnegative integers $l_1,l_2, \ldots, k_1,k_2, \ldots,$ $m_1,m_2, \ldots$ and  $i_1,i_2,\ldots \in \{1, \ldots,p\},  j_1,j_2,\ldots  \in \{1, \ldots,n\}.$ For shortness, we denote the two set of indeterminates with $\yp$ and $\xn$ respectively.
\begin{defn}
If ${\mathcal X} \equiv \nor_{p,n}(M, \Sigma, I_n),$ then 
$\Tr[(D_y {\mathcal X} D_x)(D_y {\mathcal X} D_x)^{\trasp}]$ is
the Wishart polynomial trace umbra. 
\end{defn}
This definition is well suited for quadratic forms and might be extended to different umbral matrices. As example, if $n=p$  we can consider the polynomial trace umbrae $\Tr(D_y {\mathcal V} D_x)$ or in one set of indeterminates $\Tr(D_y {\mathcal V}).$

The moments of the  Wishart polynomial trace umbra form a sequence of complete Bell polynomials as proved in the following theorem. 
\begin{thm} \label{thm1}  If ${\mathcal X} \equiv {\mathcal N}_{p,n}(M, \Sigma, I_n),$ then
$$
E\left\{ [\Tr(D_y {\mathcal X} D_x)(D_y {\mathcal X} D_x)^{\trasp}]^i\right\} = B_i \left( c_1\big(\yp,\xn \big), \ldots, c_i \big(\yp,\xn\big)\right), \,\,\, i \geq 1
$$
where $B_i$ is the $i$-th complete Bell polynomial, $c_0\big(\yp,\xn \big) =1$ and $c_k\big(\yp,\xn \big) = q_k\big(\yp,\xn \big) + \tilde{q}_k\big(\yp,\xn \big),$ for  $k \geq 1$ with
\begin{align}
q_k\big(\yp,\xn \big) & = (k-1)! \, 2^{k-1} \, \Tr[(D_x^2 \otimes \,  D_y^2 D_{\theta})^k]  \label{cum1} \\
\tilde{q}_k\big(\yp,\xn \big) & = \left\{ \begin{array}{ll}
{\rm vec}^{\trasp}(\tilde{M}) {\rm vec}(\tilde{M}), & k=1, \\
 k! \, 2^{k-1} \hbox{\rm vec}^{\trasp}(\tilde{M}) (D_x^2 \otimes  \tilde{\Sigma})^{k-1}  \hbox{\rm vec}(\tilde{M}),  &  k > 1
 \end{array} \right. \label{cum2}
\end{align}
$\tilde{M} = D_y M D_x$ and $\tilde{\Sigma}= D_y \Sigma D_y.$
\end{thm}
\noindent
\begin{proof}
From Theorem 7.8.2 of \cite{Gupta} with $A$ replaced by  $D_x^2={\rm diag}(x_1^2, \ldots, x_n^2)$ and
parametric matrix $z I_p,$ we have $f({\mathcal X } D_x^2 {\mathcal X }^{\trasp}, z)$ $= g_1(z) \exp g_2(z)$ where
by using the first of \eqref{(omsum1)}  
\begin{equation}
 g_1(z)  =  \prod_{j=1}^n \det(I_p  - 2 x_j^2 z D_{\theta})^{-1/2}  = {\rm etr} \bigg(\sum_{k \geq 1} 2^{k-1} \frac{(D_x^2 \otimes  z D_{\theta})^k}{k} \bigg) 
 \label{5.4bis} 
\end{equation}
and by using the second of \eqref{(omsum1)} and by observing that 
$[I_{np} - 2 (D_x^2 \otimes  z \Sigma)]^{-1} = 
{\rm diag} [(I_p - 2 x_1^2 z \Sigma)^{-1}, \ldots, (I_p - 2 x_n^2 z \Sigma)^{-1}]$ we have
\begin{align}
 g_2(z) & = \sum_{k \geq 0} 2^k z^{k+1} {\rm vec}^{\trasp}(M D_x) ( D_x^2 \otimes   \Sigma)^k {\rm vec}( M D_x) \\ 
 & = z  {\rm vec}^{\trasp}(M D_x) {\rm vec}(M D_x) + \sum_{k \geq 1} 2^k z^{k+1} {\rm vec}^{\trasp}(M D_x) ( D_x^2 \otimes   \Sigma)^k {\rm vec}(M D_x).  \label{5.5}
\end{align}
 Moreover $D_y {\mathcal X}  \equiv \nor_{p,n}(D_y M, D_y \Sigma D_y, I_n)$ and, taking into account  \eqref{5.4bis} and \eqref{5.5},  we have 
 $$f \big( \Tr[(D_y {\mathcal X} D_x)(D_y {\mathcal X} D_x)^{\trasp}],z \big) = \exp [g(z)-1]$$ with 
\begin{align}
g(z)  =  1 & + z [\Tr(D^2_x \otimes D_y^2 D_{\theta}) + {\rm vec}^{\trasp}(\tilde{M}) {\rm vec}(\tilde{M})] \nonumber \\
& + \sum_{k \geq 2} 2^{k-1} z^k \left( \frac{\Tr(D^2_x \otimes D_y^2 D_{\theta})^k}{k} + {\rm vec}^{\trasp} (\tilde{M}) (D_x^2 \otimes \tilde{\Sigma})^{k-1}  {\rm vec}(\tilde{M})  \right) \label{4.12}
\end{align}
with $\tilde{M}=D_y M D_x$ and $\tilde{\Sigma} = D_y \Sigma D_y.$ Therefore $g(z)$ is a formal power series with the $k$-th coefficient $c_k=c_k\big(\yp,\xn\big)$ given in \eqref{cum1} and \eqref{cum2}. The result follows as the $i$-th coefficient of $\exp[g(z)-1]$ is the $i$-th complete Bell  polynomial  $B_i$ in $(c_1, \ldots, c_i).$ 
\end{proof}
\begin{ex}{\rm For example, for $n=3$ and $p=2$ we have
$E\left\{ \Tr[(D_y {\mathcal X} D_x)(D_y {\mathcal X} D_x)^{\trasp}]\right\}  = 
\Tr(D_x^2 \otimes \,  D_y^2 D_{\theta}) + \Tr 
(\tilde{M}^{\trasp} \tilde{M}) = c_1 \big( {\ynn},  \xnn \big)$ with 
\begin{align*}
 c_1 \big(\ynn, \xnn \big)  = & 
 (x_1^2 y_1^2 + x_2^2 y_1^2 + x_3^2 y_1^2) \theta_{1}  +(x_1^2 y_2^2 + x_2^2 y_2^2 + x_3^2 y_2^2) \theta_{2} +
 y_1^2 (m_{11}^2 x_1^2  + m_{12}^2
x_2^2 + m_{13}^2 x_3^2)  \\ 
 & + y_2^2 (m_{2 1}^{2} x_1^2 +  m_{22}^2 x_2^2 + m_{23}^2 x_3^2).
\end{align*}
}
\end{ex}

Note that $g(z)-1$ in \eqref{4.12} is the cumulant g.f. of $\Tr(D_y {\mathcal X} D_x)(D_y {\mathcal X} D_x)^{\trasp}.$ Thus $\{c_k\big(\yp,\xn\big) \}_{k \geq 1}$ is the sequence of formal cumulants of the  Wishart polynomial trace umbra and an extension of the cumulant polynomials introduced in \cite{CumDin}.  As $\{c_k\big(\yp,\xn\big)\}_{k \geq 1}$ is the summation of two sequences, from the additivity property of cumulants, the Wishart polynomial trace umbra is the sum of two polynomial umbrae. In particular, $q_k\big(\yp,\xn\big)$ is the $k$-th formal cumulant of the central Wishart polynomial trace corresponding to $M=0.$ The same property  has been highlighted and discussed in \cite{Wishartmio} for the Wishart random matrix. 
\section{From the polynomial trace umbra to the e.s.f.'s\label{sec:5}}
The idea to use polynomial trace umbrae to recover $\E[\Tr_i(W)]$ relies on the following observation. As ${\mathsf X } \sim {\mathcal N}_{p,n}(0, I_p, I_n)$ is bi-unitary invariant \cite{Tulino}, then $E \big( [\Tr(\Delta_p {\mathcal X} \tilde{\Delta}_n) (\Delta_p {\mathcal X} \tilde{\Delta}_n)^{\trasp}]^i \big)= E \big( [\Tr(\Delta_p {\mathcal V} \tilde{\Delta}_n)$ $(\Delta_p {\mathcal V} \tilde{\Delta}_n)^{\trasp}]^i \big) = i! \E [\Tr_i(W)].$ 
 Therefore we might recover $i! \E[\Tr_i(W)]$ from the moments of the polynomial trace umbra $E [\Tr(D_y {\mathcal X} D_x) (D_y {\mathcal X} D_x)^{\trasp}]^i$ plugging $\deltap$
and $\deltan$ in $\yp$ and $\xn$ respectively. The following theorems give sufficient conditions on  ${\mathcal X }$ such that
\begin{equation}
i! \E[\Tr_i(W)] =  E \left\{ [\Tr(\Delta_p {\mathcal X} \tilde{\Delta}_n) (\Delta_p {\mathcal X} \tilde{\Delta}_n)^{\trasp}]^i \right\}
\label{(finale)}
\end{equation} 
still holds for $i \leq p.$
\begin{thm}\label{6.2}
If ${\mathcal X } \equiv {\mathcal N}_{p,n}(0, \Sigma, I_n),$
then \eqref{(finale)} holds with $W=W_p(n, \Sigma,0).$   
\end{thm}
\noindent
\begin{proof}
As  $(n)_i = E[(n \punt \tilde{\chi})^i]$ and $i! \Tr_i(\Sigma) = E[(\chi_1 \theta_1 + \cdots + \chi_p \theta_p)^i],$ from the first equation in \eqref{(first)} we have 
\begin{equation} 
i! \E[\Tr_i(W)] = E[(n \punt \tilde{\chi})^i (\chi_1 \theta_1 + \cdots + \chi_p \theta_p)^i ], \quad \hbox{\rm $i \leq p$}.
\label{(first1)}
\end{equation}
From Theorem \ref{thm1},  we have 
$E ([\Tr(\Delta_p {\mathcal X} $ $\tilde{\Delta}_n) (\Delta_p {\mathcal X} \tilde{\Delta}_n)^{\trasp}]^i) = E[B_i(q_1, \ldots,q_i)]$ with $q_k\big(\deltap, \deltan\big)  
=  (k-1)! \, 2^{k-1}  \Tr[(\tilde{\Delta}_n^2 \otimes \,   \Delta_p^2 D_{\theta})^k]$ for $k \geq 1$ since $\tilde{q}_k\big(\deltap, \deltan\big) =0$ for $k \geq 1.$  The explicit expression of $B_i(q_1, \ldots,q_i)$ is  
\begin{equation}
B_i(q_1, \ldots, q_i) = q_1^i + \sum_{k=1}^{i-1} \sum_{\lambda \vdash i \atop l(\lambda)=k} d_{\lambda} q_1^{r_1} \, q_2^{r_2} \times \cdots
\label{(completeBell1)}
\end{equation}
where $\lambda = (1^{r_1} 2^{r_2} \ldots)$ is a partition of $i$ in $l(\lambda)= r_1 + r_2 + \cdots$ positive integers and $d_{\lambda} = i! /(r_1! r_2! \times \cdots (1!)^{r_1} (2!)^{r_2} \times \cdots).$ Note that we write $\lambda \vdash i$ to denote that the partition is referred to the integer $i.$   Taking the evaluation of both sides in \eqref{(completeBell1)}, we have $E\big[\prod_{j} q_j^{r_j}\big] \ne 0 \,\, {\rm iff} \,\, j=1 \,\, {\rm and}\, r_1 = i$ as for $j \geq 2,$ the evaluation involves powers of delta umbrae greater than $2.$  Thus $E[ B_i(q_1, \ldots, q_i)] = E[q_1\big(\deltap, \deltan\big) ^i] = E \{ [\Tr\big(\tilde{\Delta}^2_n \otimes \Delta_p^2 D_{\theta}\big)]^i\} = E[ \Tr\big(\tilde{\Delta}^2_n\big)^i] E[\Tr\big(\Delta_p^2 D_{\theta}\big)^i].$ Note that  $\Tr\big(\tilde{\Delta}^2_n\big) \equiv n \punt \tilde{\chi}$ and $\Tr\big(\Delta_p^2 D_{\theta}\big) \equiv  \chi_1 \theta_1 + \cdots + \chi_p \theta_p.$ Hence  $E \{[\Tr\big(\Delta_p {\mathcal X} \tilde{\Delta}_n\big) \big(\Delta_p {\mathcal X} \tilde{\Delta}_n\big)^{\trasp}]^i\}$  is equal to the rhs of \eqref{(first1)} from which \eqref{(finale)} follows. 
\end{proof} 
\vskip2pt \indent
As corollary, if  $W=W_p(n,  I_p ,0)$ and  ${\mathcal X } \equiv {\mathcal N}_{p,n}(0, I_p, I_n),$ from \eqref{(finale)} one has
\begin{equation}
i!\E[\Tr_i(W)] = E[ (n \punt \tilde{\chi})^i (p \punt \chi)^i] = (n)_i (p)_i \,\,\, \hbox{\rm for} \,\,\, i \leq p \leq n.
\label{remark6.2}
\end{equation} 
\begin{thm} \label{4.5cor}
If ${\mathcal X } \equiv {\mathcal N}_{p,n}(M, \sigma^2 I_p, I_n)$ and $M$ is a $p \times n$ rectangular diagonal matrix containing  $m_1, \ldots, m_p$ in the elements with equal indices, then \eqref{(finale)}  holds with $W=W_p(n, \sigma^2 I_p ,M).$
\end{thm}
\noindent
\begin{proof}
Without loss of generality, set $\sigma^2=1.$  We first prove  
\begin{equation} 
E \left\{ \big[\Tr(\tilde{\Delta}_n^2 \otimes \Delta_p^2)+ \Tr\big(\tilde{M} \tilde{M}^{\trasp}) \big]^i \right\} =
i! \E[\Tr_i(W)]
\label{cas2eq}
\end{equation}
where $\tilde{M}=\Delta_p M \tilde{\Delta}_n.$ Then we prove that 
the rhs of \eqref{cas2eq} is the $i$-th moment of $\Tr\big(\Delta_p {\mathcal X} \tilde{\Delta}_n\big) \big(\Delta_p {\mathcal X} \tilde{\Delta}_n\big),$  that gives \eqref{(finale)}. To prove \eqref{cas2eq} note that 
the umbral polynomial $\big[\Tr\big(\tilde{\Delta}_n^2 \otimes \Delta_p^2\big)+ \Tr\big(\tilde{M} \tilde{M}^{\trasp}\big) \big]^i$ has the same evaluation of
\begin{equation}
\big[ (\chi_1 + \cdots + \chi_p)(\tilde{\chi}_1 + \cdots + \tilde{\chi}_n) + (\chi_1 \tilde{\chi}_1 m^2_1 + \cdots + \chi_p \tilde{\chi}_p m^2_p) \big]^i.
\label{(aaa1)}
\end{equation}
The aim is to recover the second equation in \eqref{(first)}, by using  the binomial expansion in \eqref{(aaa1)} and then by applying $E.$
 Notice that $E[(\chi_1 \tilde{\chi}_1 m^2_1 + \cdots + \chi_p \tilde{\chi}_p m^2_p)^i]$ $ = i!  e_i (m^2_1, \ldots, m^2_p)$ and $E\{[(\chi_1 + \cdots + \chi_p)(\tilde{\chi}_1 + \cdots + \tilde{\chi}_n)]^i\} = (p)_i (n)_i$ from   \eqref{remark6.2}. The subsequent step is to evaluate the cross terms $(\chi_1 + \cdots + \chi_p)^{i-j} (\tilde{\chi}_1 + \cdots + \tilde{\chi}_n)^{i-j} (\chi_1 \tilde{\chi}_1 m^2_1 + \cdots + \chi_p \tilde{\chi}_p m^2_p)^{j}/[(i-j)!j!].$ By using the multinomial expansion and after some algebra, the monomials with not zero evaluation are 
\begin{equation}
(i-j)! \chi_1^{t_1} \timesmall \cdots \timesmall \, \chi_p^{t_p} \tilde{\chi}_1^{s_1} \timesmall \cdots \timesmall \, \tilde{\chi}_n^{s_n} \chi_1^{l_1} \tilde{\chi}_1^{l_1} \timesmall \cdots \timesmall \, \chi_p^{l_p} \tilde{\chi}_p^{l_p} m_1^{2 l_1} \timesmall \cdots \timesmall \, m_p^{2 l_p}
\label{addendi}
\end{equation}
where
$t_1, \ldots, t_p, s_1, \ldots, s_n, l_1, \ldots, l_p \in \{0,1\}$ are such that $t_1 + \cdots + t_p=i-j, s_1 + \cdots + s_n=i-j,  l_1 + \ldots + l_p=j$ and not allowing repetitions of the same singleton umbrae.  Thus in \eqref{addendi} there are $i$ distinct singleton umbrae choosen among $\{\chi_1, \ldots, \chi_p\}$ and $i$ distinct singleton umbrae choosen among $\{\tilde{\chi}_1, \ldots, \tilde{\chi}_n\},$ but with $j$ indexes fixed in both. In particular grouping togheter the monomials in $m_1^2, \ldots, m_p^2$ we recover $(i-j)!$  times the summation
\begin{equation}
\sum_{1 \leq k_1 < \cdots < k_j \leq p} \chi_{k_1} \tilde{\chi}_{k_1} \timesmall \cdots \timesmall \,  \chi_{k_j} \tilde{\chi}_{k_j} m^2_{k_1} \timesmall \cdots \timesmall \, m^2_{k_j} \sum_{t_1 < \cdots < t_{i-j} \atop s_1 < \cdots < s_{i-j}} \chi_{t_1} \timesmall \cdots \timesmall \, \chi_{t_{i-j}}  \tilde{\chi}_{s_1} \timesmall \cdots \timesmall \,\tilde{\chi}_{s_{i-j}}
\label{addendi1}
\end{equation}
where $t_1, \ldots, t_{i-j} \in [p]-\{k_1, \ldots, k_j\}$ and $s_1, \ldots, s_{i-j} \in [n]-\{k_1, \ldots, k_j\}.$  Note that,  fixed the $j$ products $\chi_{k_1} \tilde{\chi}_{k_1}, \ldots, \chi_{k_j} \tilde{\chi}_{k_j}$ (the order does not matter)  in  \eqref{addendi1}, the  singleton umbrae in the second summation can be chosen in $\binom{p-j}{i-j}$ ways among $\{\chi_1, \ldots, \chi_p\}$ and $\binom{n-j}{i-j}$ ways among  $\{\tilde{\chi}_1, \ldots, \tilde{\chi}_n\}.$ Moreover the singleton umbrae in the outer summation of \eqref{addendi1} are uncorrelated with the singleton umbrae of the inner summation and so the evaluation of  \eqref{addendi1} is the same if we relabel the first ones with $\chi'_{k_1} \tilde{\chi}'_{k_1} \timesmall \cdots \timesmall \, \chi'_{k_j} \tilde{\chi}'_{k_j}$ and indexed the second ones with the elements of $[i-j].$  Equation 
\eqref{cas2eq} follows after some algebra, taking into account that 
\begin{eqnarray}
E \bigg(\sum_{1 \leq t_1 < \cdots < t_{i-j} \leq i-j \atop 1 \leq _1 < \cdots < s_{i-j} \leq i-j} \chi_{t_1} \timesmall \cdots \timesmall \,  \chi_{t_{i-j}}  \tilde{\chi}_{s_1} \timesmall \cdots \timesmall \,  \tilde{\chi}_{s_{i-j}} \bigg)  \!\!\!\!  & = & \!\!\!\!  \binom{n-j}{i-j} \binom{p-j}{i-j}, \label{(conjecture)}\\
E \bigg(\sum_{1 \leq k_1 < \cdots < k_j \leq p} \tilde{\chi}_{k_1} \tilde{\chi}'_{k_1} \timesmall \cdots \timesmall \,   \tilde{\chi}_{k_j} \tilde{\chi}'_{k_j} m^2_{k_1} \timesmall \cdots \timesmall \,  m^2_{k_j}\bigg) \!\!\!\!  & = &\!\!\!\!   e_j(m^2_1, \ldots, m^2_p). \nonumber
\end{eqnarray} 
Now let us compute the
moments of  $\Tr(\Delta_p {\mathcal X} \tilde{\Delta}_n) (\Delta_p {\mathcal X} \tilde{\Delta}_n)^{\trasp}$ using Theorem  
\ref{thm1}. For $k \geq 1,$ we have $c_k \big(\deltap, \deltan \big) = q_k \big(\deltap, \deltan \big) + \tilde{q}_k \big(\deltap, \deltan \big)$  with
\begin{eqnarray*}
q_k \big(\deltap, \deltan \big) & = &  (k-1)! \, 2^{k-1}  \Tr[\big(\tilde{\Delta}_n^2 \otimes \,  \Delta_p^2\big)^k], \\
\tilde{q}_k \big(\deltap, \deltan \big) & = & \left\{ \begin{array}{ll}
{\rm vec}^{\trasp}(\tilde{M}) {\rm vec}(\tilde{M}), & k=1, \\
 k! \, 2^{k-1}  \hbox{\rm vec}^{\trasp}(\tilde{M}^{\trasp} ) (\tilde{\Delta}_n^2 \otimes  \Delta_p^2)^{k-1}  \hbox{\rm vec}(\tilde{M}),   & k > 1
 \end{array} \right. 
\end{eqnarray*}
and $\tilde{M} = \Delta_p M \tilde{\Delta}_n.$   By using the same arguments of Theorem  \ref{6.2}, we have  $E[B_i(c_1, \ldots, c_i)] = E[c_1 \big(\deltap, \deltan \big)^i] = E \left\{ \big[\Tr \big(\tilde{\Delta}_n^2 \otimes \Delta_p^2 \big)+ \Tr\big(\tilde{M} \tilde{M}^{\trasp}) \big]^i \right\}$ from which \eqref{(finale)} follows.
\end{proof}
\begin{rem}
{\rm Notice that \eqref{addendi1}, together with \eqref{(conjecture)}, corresponds to the matrix $L_{i,j}$
conjectured by de Waal in \cite{Dewaal}.}
\end{rem}
\begin{cor} \label{5.4}
If ${\mathcal X } \equiv {\mathcal N}_{p,n}(M, \sigma^2 I_p, I_n)$ and $M =P D_m Q^{\trasp},$ with $P$ and $Q$ orthogonal matrices of order $p$ and $n$ respectively and $D_m$ a $p \times n$ rectangular diagonal matrix containing  $m_1, \ldots, m_p$ in the elements with equal indices, then 
\begin{equation}
i! \E[\Tr_i(W)] =  E \left( [\Tr(\Delta_p P^{\trasp} {\mathcal X } Q\tilde{\Delta}_n) (\Delta_p P^{\trasp} {\mathcal X } Q \tilde{\Delta}_n)^{\trasp}]^i \right)\!,
 \,\, i \leq p
\label{(finale2)}
\end{equation} 
where $W=W_p(n, \sigma^2 I_p ,M).$  
\end{cor}
\noindent
\begin{proof}
The result follows from Theorem \ref{4.5cor} with ${\mathcal X }$ replaced by $P^{\trasp} {\mathcal X } Q \equiv {\mathcal N}_{p,n}(D_m, \sigma^2 I_p, I_n)$ as we have 
$$E \left( [\Tr(\Delta_p P^{\trasp} {\mathcal X} Q \tilde{\Delta}_n) (\Delta_p P^{\trasp} {\mathcal X} Q  \tilde{\Delta}_n)^{\trasp}]^i \right) = i! \E[\Tr_i(P^{\trasp} {\mathsf X} {\mathsf X}^{\trasp} P)].$$
\end{proof}
\begin{cor} \label{5.9}
If  ${\mathcal X } \equiv {\mathcal N}_{p,n}(M, \Sigma , I_n)$  and $\Sigma^{-1/2} M = P  D_m Q^{\trasp},$ with $P$ and $Q$ orthogonal matrices of order $p$ and $n$ respectively and $D_m$ a $p \times n$ rectangular diagonal matrix containing  $m_1, \ldots, m_p$ in the elements with equal indices, then 
\begin{equation}
p! \E[\Tr_p(W)] = \det(\Sigma) E[ (\Tr(\Delta_p P^{\trasp} {\mathcal X } Q \tilde{\Delta}_n)(\Delta_p P^{\trasp} {\mathcal X } Q \tilde{\Delta}_n)^{\trasp})^p ]
\label{(finale1)}
\end{equation}
where $W=W_{p}(n, \Sigma, M).$  
\end{cor}
\noindent
\begin{proof}
As $\Sigma^{-1/2}$ is a symmetric square root matrix, we have $\E[\Tr_p(\Sigma^{-1/2} W \Sigma^{-1/2})] =\E[ \det(\Sigma^{-1/2} W$ $ \Sigma^{-1/2})] = \det(\Sigma^{-1})$ 
$\E[\det(W)] = \det(\Sigma^{-1})  \E[\Tr_p(W)]$ with 
$\Sigma^{-1/2} W \Sigma^{-1/2} = W_{p}(n, I_p, \Sigma^{-1/2} M)$ \cite{Gupta}. The result follows from Corollary \ref{5.4} applied to  $\Sigma^{-1/2} W \Sigma^{-1/2}.$ 
\end{proof}  
\vskip2pt \indent
To express $\E[\Tr_i(W)]$ in terms of umbral polynomial traces, notice that $\Tr_i(W) = \sum_{j(i)} \det W_{j(i)},$ with $W_{j(i)}$
the Wishart random matrix with covariance $\Sigma_{j(i)},$  
non-centrality matrix $\Sigma^{{\scriptscriptstyle -1}}_{j(i)} (M M^{\trasp})_{j(i)},$ where $j(i) = \left\{ j_1, \ldots, j_i \right\}$ $\subset \{1, \ldots, p\}$ denotes the principal submatrix corresponding to the $j_1 < \cdots < j_i$-th rows and the $j_1  < \cdots < j_i$-th  columns of $W.$ 
From \eqref{(finale1)}, we get 
$$
\E[\Tr_i(W)] = \sum_{j(i)} \det(\Sigma_{j(i)}) E \left\{ \Tr_i\big[ \tilde{{\mathcal X}}_{j(i)}  \tilde{{\mathcal X}}_{j(i)}^{\trasp} \big] \right\},  \,\, i \leq p,
$$
with $\tilde{{\mathcal X}}_{j(i)} = \Delta_i P_i  {\mathcal X}_i Q_i^{\trasp} \Delta_n$ where ${\mathcal X}_i$ is a normal umbral matrix such that ${\mathcal X}_i  {\mathcal X}_i ^{\trasp}$ is the umbral counterpart of $W_{j(i)}$ and $P_i$ and $Q_i$ are squared matrices of order $i$ and $n$ respectively having the same properties of $P$ and $Q$ given in Corollary \ref{5.9} with respect to the non-centrality matrix $ \Sigma^{{\scriptscriptstyle -1}}_{j(i)} (M M^{\trasp})_{j(i)}.$
\section{Conclusions}
The purpose of this paper is to show how to use the umbral operator and the symbolic calculus to recover the e.s.f.'s in the Wishart matrix latent roots. To achieve this goal, we have introduced a new class of polynomials related to a random matrix trace and named polynomial traces. When the delta umbrae are plugged in the indeterminates and the corresponding umbral polynomials are evaluated through the operator $E,$ all the monomials not contributing in the e.s.f.'s delete. For some special case, see \eqref{cas2eq}, the e.s.f.'s can be recovered plainly through a binomial expansion of the sum of two polynomial traces, one related to the central component of the Wishart matrix and the other related to the mean of $W.$

The paper leaves open many questions which are in the agenda of future research.  First, this symbolic calculus might be applied to a wider class of variate distributions, see \cite{Mijares} and \cite{Pillai}.  Secondly, similar computations might be carried out for other families of symmetric polynomials, see \cite{Wishartmio}. 
In addition, the achieved results suggest to exploit the connection between  zonal polynomials and the polynomial traces, as manageable expressions of zonal polynomials are not yet known. Last but not least, the Wishart distribution exists for a larger range of the shape parameter, see \cite{Graczyk, LetacMassam} and \cite{Mayerhofer}. The method here proposed works for any positive integer $n$ as it relies on m.g.f.'s and not involves probability density functions. Taking into account the results of \cite{Wishartmio} (Proposition $8$), the method might be extended to more general shape parameter due to the infinitely divisibility of a Wishart distribution. 
\section*{Appendix}
Joint moments  of a square random matrix $Y$ of order $p$ are $\mu(Y)(\tau) =  \prod_{c \in C(\tau)} \Tr\left(Y^{{\mathfrak l}(c)}\right)$ and depend only on the cycle class  $C(\tau)$ of the permutation $\tau \in   {\mathfrak S}_i,$ see \cite{CC}. In particular we have $\E \left[ \mu(Y)({\mathfrak e}) \right] = \E[\Tr(Y)^i]$ where ${\mathfrak e}$ is the identity permutation, and $\mu(I_p)(\tau) = p^{|C(\tau)|}.$ 
To recover \eqref{(symPS1bis)}, write out the $i$-th e.s.f. in \eqref{(elfun)} through the $i$-th Bell polynomial as follows
\begin{equation}
i! e_i(y_1, \ldots, y_p)  =  \sum_{\lambda \vdash i} d_{\lambda} \prod_{k} \big[(-1)^{k-1} (k-1)! \Tr(D_y^k)\big]^{r_k} = 
\sum_{\lambda \vdash i}  (-1)^{i-\ell(\lambda)} d_{\lambda} \prod_{k}  \big[(k-1)! \Tr(D_y^k)\big]^{r_k} 
\label{(symPS)}
\end{equation}
where the summation is over all partitions $\lambda = (1^{r_1}2^{r_2} \ldots)$ of the integer $i = r_1 + 2 r_2 + \cdots$ in $\ell(\lambda)=r_1 + r_2 + \cdots$ parts and  $d_{\lambda} = i!/[(1!)^{r_1} r_1! (2!)^{r_2} r_2! \times \cdots].$  Recall that the cycle class  of a permutation $\tau$ with $r_1$ cycles of length $1,$ $r_2$
cycles of length $2,$ and so on, is the integer partition $\lambda = (1^{r_1} 2^{r_2} \ldots)   \vdash i,$ whose number of parts $\ell(\lambda)$ is equal to the number of cycles $|C(\tau)|.$ By observing that $d_{\lambda} = {\mathfrak s}_{\lambda}/[(1!)^{r_2} (2!)^{r_3} \times \cdots],$
where ${\mathfrak s}_{\lambda}$ is the number of permutations 
$\tau \in  {\mathfrak S}_i$ of cycle class $\lambda = (1^{r_1}2^{r_2} \ldots) \vdash i \leq p,$ from \eqref{(symPS)} we have 
\begin{equation}
i! e_i(y_1, \ldots, y_p)  =  \sum_{\lambda \vdash i} {\mathfrak s}_{\lambda}  (-1)^{i-\ell(\lambda)} \prod_{k} \big[\Tr(D_y^k)\big]^{r_k} = \sum_{\tau \in {\mathfrak S}_i} (-1)^{i-|C(\tau)|} \mu(D_y)(\tau)
\label{(symPSbis)}
\end{equation}
where the last equality follows by indexing the summation with respect to permutations. Hence \eqref{(symPS1bis)} follows from \eqref{(symPSbis)} by replacing $y_1, \ldots, y_p$ with $Y_1, \ldots, Y_p.$


\begin{thebibliography}{10}

\bibitem{Taqqu}
F.~Avram and M.~S. Taqqu.
\newblock Noncentral limit theorems and {A}ppell polynomials.
\newblock {\em Ann. Probab.}, 15(2):767--775, 1987.

\bibitem{CC}
M.~Capitaine and M.~Casalis.
\newblock Cumulants for random matrices as convolutions on the symmetric group.
\newblock {\em Probab. Theory Related Fields}, 136(1):19--36, 2006.

\bibitem{Dewaal}
D.~J. de~Waal.
\newblock On the expected values of the elementary symmetric functions of a
  noncentral {W}ishart matrix.
\newblock {\em Ann. Math. Statist.}, 43:344--347, 1972.

\bibitem{DiBucchianico}
A.~Di~Bucchianico.
\newblock {\em Probabilistic and analytical aspects of the umbral calculus},
  volume 119 of {\em CWI Tract}.
\newblock Stichting Mathematisch Centrum, Centrum voor Wiskunde en Informatica,
  Amsterdam, 1997.

\bibitem{DiNardo}
E.~Di~Nardo.
\newblock Symbolic calculus in mathematical statistics: a review.
\newblock {\em S\'{e}m. Lothar. Combin.}, 67:Art. B67a, 72, 2011/12.

\bibitem{Wishartmio}
E.~Di~Nardo.
\newblock On a symbolic representation of non-central {W}ishart random matrices
  with applications.
\newblock {\em J. Multivariate Anal.}, 125:121--135, 2014.

\bibitem{CumDin}
E.~Di~Nardo.
\newblock On multivariable cumulant polynomial sequences with applications.
\newblock {\em J. Algebr. Stat.}, 7(1):72--89, 2016.

\bibitem{annali}
E.~Di~Nardo, P.~McCullagh, and D.~Senato.
\newblock Natural statistics for spectral samples.
\newblock {\em Ann. Statist.}, 41(2):982--1004, 2013.

\bibitem{Graczyk}
Piotr Graczyk, Jacek Ma\l~ecki, and Eberhard Mayerhofer.
\newblock A characterization of {W}ishart processes and {W}ishart
  distributions.
\newblock {\em Stochastic Process. Appl.}, 128(4):1386--1404, 2018.

\bibitem{Gupta}
A.~K. Gupta and D.~K. Nagar.
\newblock {\em Matrix variate distributions}, volume 104 of {\em Chapman \&
  Hall/CRC Monographs and Surveys in Pure and Applied Mathematics}.
\newblock Chapman \& Hall/CRC, Boca Raton, FL, 2000.

\bibitem{Kollo}
T.~Kollo and D.~von Rosen.
\newblock {\em Advanced Multivariate Statistics with Matrices}, volume 579 of
  {\em Mathematics and Its Applications (New York)}.
\newblock Springer, Dordrecht, 2005.

\bibitem{PARUCHURI}
P.R. Krishnaiah.
\newblock {\em Some Recent Developments on Real Multivariate Distributions},
  volume~1 of {\em Developments in Statistics}.
\newblock Elsevier, 1978.

\bibitem{Letac}
G.~Letac and H.~Massam.
\newblock All invariant moments of the {W}ishart distribution.
\newblock {\em Scand. J. Statist.}, 31(2):295--318, 2004.

\bibitem{LetacMassam}
G\'{e}rard Letac and H\'{e}l\`ene Massam.
\newblock The {L}aplace transform {$(\det s)^{-p}\exp{\rm tr}(s^{-1}w)$} and
  the existence of non-central {W}ishart distributions.
\newblock {\em J. Multivariate Anal.}, 163:96--110, 2018.

\bibitem{Macdonald}
I.~G. Macdonald.
\newblock {\em Symmetric functions and {H}all polynomials}.
\newblock Oxford Classic Texts in the Physical Sciences. The Clarendon Press,
  Oxford University Press, New York, second edition, 2015.

\bibitem{Mathai}
A.~M. Mathai, Serge~B. Provost, and Takesi Hayakawa.
\newblock {\em Bilinear Forms and Zonal Polynomials}, volume 102 of {\em
  Lecture Notes in Statistics}.
\newblock Springer-Verlag, New York, 1995.

\bibitem{Mayerhofer}
Eberhard Mayerhofer.
\newblock On {W}ishart and noncentral {W}ishart distributions on symmetric
  cones.
\newblock {\em Trans. Amer. Math. Soc.}, 371(10):7093--7109, 2019.

\bibitem{Mijares}
Tito~A. Mijares.
\newblock The moments of elementary symmetric functions of the roots of a
  matrix in multivariate analysis.
\newblock {\em Ann. Math. Statist.}, 32:1152--1160, 1961.

\bibitem{Pillai}
K.~C.~S. Pillai and G.~M. Jouris.
\newblock On the moments of elementary symmetric functions of the roots of two
  matrices.
\newblock {\em Ann. Inst. Statist. Math.}, 21:309--320, 1969.

\bibitem{Twelve}
G.-C. Rota.
\newblock Twelve problems in probability no one likes to bring up.
\newblock In {\em Algebraic combinatorics and computer science}, pages 57--93.
  Springer Italia, Milan, 2001.

\bibitem{SIAM}
G.-C. Rota and B.~D. Taylor.
\newblock The classical umbral calculus.
\newblock {\em SIAM J. Math. Anal.}, 25(2):694--711, 1994.

\bibitem{Saw}
J.~G. Saw.
\newblock Expectation of elementary symmetric functions of a {W}ishart matrix.
\newblock {\em Ann. Statist.}, 1:580--582, 1973.

\bibitem{Shah}
B.~K. Shah and C.~G. Khatri.
\newblock Proof of conjectures about the expected values of the elementary
  symmetric functions of a noncentral {W}ishart matrix.
\newblock {\em Ann. Statist.}, 2:833--836, 1974.

\bibitem{Smith}
Peter~J. Smith and Lee~M. Garth.
\newblock Distribution and characteristic functions for correlated complex
  {W}ishart matrices.
\newblock {\em J. Multivariate Anal.}, 98(4):661--677, 2007.

\bibitem{Stanley}
R.~P. Stanley.
\newblock {\em Enumerative combinatorics. {V}olume 1}, volume~49 of {\em
  Cambridge Studies in Advanced Mathematics}.
\newblock Cambridge University Press, Cambridge, second edition, 2012.

\bibitem{Tulino}
A.~M Tulino and S.~Verdu.
\newblock {\em Random Matrices and Wireless Communications}, volume~1 of {\em
  Foundations and Trends in Communications and Information Theory}.
\newblock now Publishers Inc, 2004.

\end{thebibliography}
\end{document}